\documentclass[11pt]{article}

\usepackage{ulem}
\usepackage{framed}
\usepackage[%
 hypertexnames=false,
 pagebackref,
 setpagesize=false,%
 colorlinks=true,%
 linkcolor=blue,
 citecolor=red,
 pdftitle={},%
 pdfsubject={},%
 pdfauthor={},%
 pdfkeywords={}%
]{hyperref}

\usepackage{amsthm}
\usepackage{amssymb}
\usepackage{amscd}
\usepackage{amsmath}
\usepackage[all]{xy}
\usepackage[top=30truemm,bottom=30truemm,left=25truemm,right=25truemm]{geometry}
\usepackage{comment}
\usepackage{algorithmic,algorithm}
\usepackage{here}
\usepackage{graphicx}
\usepackage{color} %
\usepackage{enumerate} %
\usepackage{listliketab}
\usepackage{multirow}
\usepackage{url}


\makeatletter
  
  \@addtoreset{algorithm}{subsection}
\makeatother

\theoremstyle{definition}

\theoremstyle{definition}

\theoremstyle{plain}
\newtheorem{theo}{Theorem}

\theoremstyle{plain}

\theoremstyle{plain}

\theoremstyle{plain}

\theoremstyle{plain}
\newtheorem{thm}{Theorem}[subsection]

\theoremstyle{definition}
\newtheorem{ex}[thm]{Example}

\theoremstyle{definition}

\theoremstyle{definition}

\theoremstyle{definition}
\newtheorem{defin}[thm]{Definition}

\theoremstyle{definition}
\newtheorem{rem}[thm]{Remark}

\theoremstyle{plain}
\newtheorem{prop}[thm]{Proposition}

\theoremstyle{plain}
\newtheorem{lem}[thm]{Lemma}

\theoremstyle{plain}
\newtheorem{cor}[thm]{Corollary}

\theoremstyle{definition}

\theoremstyle{definition}

\theoremstyle{definition}

\theoremstyle{definition}

\theoremstyle{definition}

\theoremstyle{definition}
\newtheorem{al}[thm]{Algorithm}

\numberwithin{equation}{subsection}

\def\deg{{\rm deg}}

\def\rank{{\rm rank}}

\def\Fvtex{\mathbb{F}}
\def\deg{\mathrm{deg}}
\def\rank{\mathrm{rank}}

\newcommand{\Proj}{\operatorname{Proj}}

\newcommand{\PGL}{\operatorname{PGL}}

\usepackage[pagewise]{lineno}

\makeatletter 
\long\def\@makefntext#1{\parindent 1em\noindent 
\@hangfrom{\hbox to 1.8em{\hss $^{\@thefnmark}$}}#1}
\makeatother
\makeatletter
\def\@seccntformat#1{\csname the#1\endcsname. }
\makeatother
\makeatletter
\renewcommand\section{\@startsection {section}{1}{\z@}%
 {-3.5ex \@plus -1ex \@minus -.2ex}%
 {2.3ex \@plus.2ex}%
 {\normalfont\large\bfseries}}
\makeatother

\newcommand{\keywords}[1]{\textbf{{Keywords: }} #1}
\newcommand{\MSC}[1]{\textbf{{2010 Mathematical Subject Classification: }} #1}
\begin{document}

\title
{\bf Representation of non-special
curves of genus 5\\ as 
plane sextic curves
and its application to\\ finding curves with many rational points}
\author
{Momonari Kudo\thanks{Faculty of Information Engineering, Fukuoka Institute of Technology, Japan}
\ and Shushi Harashita\thanks{Graduate School of Environment and Information Sciences, Yokohama National University, Japan}}

\maketitle

\begin{abstract}
In algebraic geometry, it is important to provide effective para\-metrizations
for families of curves, both in theory and in practice. In this paper,
we present such an effective parametrization for the moduli of
genus-$5$ curves that are neither hyperelliptic nor trigonal. Subsequently,
we construct an algorithm for a complete enumeration of non-special genus-$5$
curves having more rational points than a specified bound, where
``non-special curve'' means that the curve is non-hyperelliptic and
non-trigonal with mild singularities of the associated sextic model that
we propose. As a practical application, we implement this algorithm using
the computer algebra system MAGMA, specifically for curves over the prime
field of characteristic $3$.
\end{abstract}

\footnote[0]{\keywords Algebraic curves, Rational points, Curves of genus five, Non-hyperelliptic and non-trigonal curves\\
\MSC 14G05, 14G15, 
14H10, 14H45, 14H50, 14Q05, 68W30}

\section{Introduction}
\label{sec:Intro}

Let $K$ be a field and let $\mathbb{P}^{n}_{K}$ denote the projective
$n$-space over $K$. Parameterizing the space of curves over $K$ of given
genus is of significance in algebraic geometry, number theory, and arithmetic
geometry. For the hyperelliptic case, it is well-known that any hyperelliptic
curve over $K$ of genus $g$ is the normalization of $y^{2} = f(x)$ for
a square-free polynomial $f(x) \in K[x]$ of degree $2 g+1$ or
$2 g+2$, if the characteristic of $K$ is odd. Over the algebraic closure
$\overline{K}$, we can make $f(x)$ monic with $\deg f=2g+1$, eliminate
one coefficient and normalize another to yield $2g-1$ parameters. The
dimension of the moduli space is $2g-1$, resulting in effective parametrizations.
However, for a family of non-hyperelliptic curves, it is not easy in general
to find their defining equations with the fewest possible parameters,
by choosing a suitable model of the family. Note that for $g\ge 2$, a curve
is non-hyperelliptic if and only if the canonical sheaf is very ample (cf.\ \cite[Chap. IV, Prop. 5.2]{Har}).
For a non-hyperelliptic curve, its image of the canonical embedding is
called a \textit{canonical curve}.

For genus $3$, a suitable parametrization is reported in \cite{LRRS}.
For genus $4$, a canonically embedded curve is the complete intersection of a quadric and a cubic in ${\mathbb P}^{3}_{K}$.
In \cite{KH16} and \cite{KH17}, we discussed reductions of the space of pairs of quadratic forms and cubic forms, where reduction means representing the space by as few parameters as possible up to the action of the projective general linear group $\PGL _{4}(K)$ of degree $4$.

The next target is the case of $g=5$. In this case, it is known that there
are two types of non-hyperelliptic curves; trigonal or non-trigonal.
For the trigonal case, we previously studied a quintic model in
$\mathbb{P}^{2}_{K}$ of a trigonal curve of genus $5$ in order to enumerate
superspecial ones over small finite fields~\cite{trigonal}. Under some
assumptions, we presented reductions of quintic forms defining trigonal
curves of genus $5$; see \cite[Section 3]{trigonal} for details. The remaining
case considers that the curve is non-hyperelliptic and non-trigonal,
and it is known that the curve is realized as the complete intersection
of three quadratic hypersurfaces in $\mathbb{P}^{4}_{K}$ (cf.\ \cite[Chap.~IV, Example 5.5.3 and Exercises 5.5]{Har}).
This case is significantly more difficult than the hyperelliptic and
trigonal cases, since we need many parameters to define three quadratic
forms in five variables. Although it is natural to reduce the parameters
by using the natural action by $\PGL _{5}(K)$, it would be extremely hard
to give an efficient reduction in this manner, since the group
$\PGL _{5}(K)$ is large and complicated.

In this paper, we propose an effective parametrization for the space
of non-hyperelliptic and non-trigonal curves of genus $5$. Specifically,
it will be proved that any non-hyperelliptic and non-trigonal curve
$C$ of genus $5$ is bi-rational to a sextic $C'=V(F)$ in
$\mathbb{P}^{2}_{K}$. Note that experimentally in most cases, $C'$ has
five double points. In such a case, we say that $C$ is \textit{non-special};
see {Definition~\ref{def:generic}} (and {Lemma~\ref{lem:nonsp}}) for the rigorous
definition. We also show that the dimension of the space of non-special
curves with fixed singularities is at most $12$, which is precisely the
dimension of the moduli space of curves of genus $5$. Specifically, the
coefficients of $F$ have linear expressions in $12$ parameters, and the
expressions can all be computed.

As an application of this parametrization, we present an algorithm to enumerate
non-special curves $C$ over a finite field $K$ of genus $5$ with a prescribed
number of $L$-rational points, where $L$ is a finite extension field of
$K$. Executing the algorithm for $K = \mathbb{F}_{3}$ and
$L = \mathbb{F}_{9}$ on MAGMA, we obtain the following theorem:

\begin{theo}%
\label{thm:main}
The maximal number of $\Fvtex _{9}$-rational points on a non-special curve
$C$ of genus $5$ defined over $\Fvtex _{3}$ (and not over $\Fvtex _{9}$) is
$32$. Moreover, there are exactly four $\Fvtex _{9}$-isogeny classes of Jacobian
varieties of non-special curves $C$ of genus $5$ over $\Fvtex _{3}$ with 32
$\Fvtex _{9}$-rational points, whose Weil polynomials are
\begin{enumerate}[\textup{(4)}]
\item[\textup{(1)}] $(t^{2} + 2 t + 9)(t^{2} + 5 t + 9)^{4}$,
\item[\textup{(2)}]
$(t + 3)^{2}(t^{4} + 8 t^{3} + 32 t^{2} + 72 t + 81)^{2}$,
\item[\textup{(3)}] $(t + 3)^{4}(t^{2} + 2 t + 9)(t^{2} + 4 t + 9)^{2}$,
\item[\textup{(4)}] $(t + 3)^{6}(t^{2} + 2 t + 9)^{2}$.
\end{enumerate}
In Section~\ref{sec:new5}, examples of non-special curves $C$ over
$\Fvtex _{3}$ with $\# C (\mathbb{F}_{9}) = 32$ will be given.
\end{theo}

As in the website manypoints.org~\cite{ManyPoints}, the maximal number
of $\#C(\Fvtex _{9})$ of curves $C$ of genus $5$ over $\Fvtex _{9}$ is unknown,
but is known to belong between $32$ and $35$ (this upper bound is due to Lauter~\cite{Lauter}). On the website, three examples of $C$ with 32
$\Fvtex _{9}$-rational points are listed. The above theorem gives at least
one new example. More concretely, the Weil polynomial of Fischer's example
$(x^{4} + 1)y^{4} + 2x^{3}y^{3} + y^{2} + 2xy + x^{4} + x^{2} = 0$ is
$(t^{2} + 2 t + 9)(t^{2} + 5 t + 9)^{4}$. In fact, this curve appears in
our computation, since by dividing the example by $x^{4}y^{4}$ we obtain
the sextic form
$1+X^{4} + 2XY + X^{4}Y^{2} + 2X^{3}Y^{3} + Y^{4} + X^{2}Y^{4}$ (having
distinct 5 singular points) with $X:=1/x$ and $Y:=1/y$. The example of
\cite{Ramos-Ramos} (submitted by Ritzenthaler to the site)
$y^{8} = a^{2} x^{2}(x^{2}+a^{7})$ with $a^{2}+a+2=0$ has the Weil polynomial
$(t + 3)^{6}(t^{2} + 2 t + 9)^{2}$. This curve is defined over
$\Fvtex _{9}$, but the above theorem finds a curve over the prime field
$\Fvtex _{3}$ with the same Weil polynomial. From this theorem, we see that
if one wants to find curves of genus $5$ over $\mathbb{F}_{9}$ with
$\#C(\mathbb{F}_{9})>32$, one needs to search those not defined over
$\mathbb{F}_{3}$ or curves whose sextic models have more complex singularities.

The explicit parametrization and algorithm presented in this paper have
the potential to yield valuable applications in both theoretical understanding
and computational exploration. One possible application is the classification
of non-hyperelliptic and non-trigonal curves of genus $5$ with specific
invariants (Hasse-Witt rank and Ekedahl-Oort type, etc.). Some open problems
will be summarized in Section~\ref{sec:new6}. One of our future works aims
to enumerate superspecial non-special curves of genus $5$ over finite fields.

\subsection*{Acknowledgments}
The authors thank the anonymous referees for their comments and suggestions, which have helped the authors significantly improve the paper.
The contents of this paper were presented at Effective Methods in Algebraic Geometry (MEGA 2021).
The authors also thank the anonymous referees assigned by the organization of this conference for their
comments and suggestions. All of them are taken into account for improving the presentation of the
paper. The authors are also sincerely grateful to Christophe Ritzenthaler for his helpful comments.
This work was supported by JSPS Grant-in-Aid for Scientific Research (C) 17K05196 and 21K03159, and JSPS Grant-in-Aid for Young Scientists 20K14301 and 23K12949. We would like to thank Editage (\url{www.editage .jp}) for English language editing.

\section{Non-hyperelliptic and non-trigonal curves of genus 5}
\label{sec:2}

In this section, we study genus-$5$ curves which are neither hyperelliptic
nor trigonal. In order to parameterize these curves, we propose using the
realization of these curves as (singular) plane sextic curves. This method
is much more effective than using the realization by the complete intersection
of three quadratic hypersurfaces in ${\mathbb P}^{4}$. Here we give an explicit
construction of the sextic model for our purpose. Although one may find
a more conceptional way to construct such a sextic model in a general curve
in \cite[Chap.~VI, Exercises, F-24 on p.~275]{ACGH}, it requires some assumptions
and does not explain all of this section.

\subsection{Sextic models}
\label{subsec:sextic}

The canonical model of a non-hyperelliptic and non-trigonal curve of genus
$5$ is the intersection of three quadrics in ${\mathbb P}^{4}$. Let
$C$ be such a curve, say
$V(\varphi _{1}, \varphi _{2}, \varphi _{3})$ in ${\mathbb P}^{4}$ for
three quadratic forms $\varphi _{1}$, $\varphi _{2}$, and
$\varphi _{3}$ in $x_{0},x_{1},x_{2},x_{3},x_{4}$. A sextic model associated
to $C$ is obtained by additional data: two points $P$ and $Q$ on $C$. We
find a short explanation of this construction in
\cite[Chap.~IV, Example 5.5.3]{Har}, but that appears away from the context
of looking at the space of curves. The sextic model is defined by the scheme-theoretic
image of the birational map defined by the (incomplete) linear system defined
by the three-dimensional subspace consisting of elements of
$H^{0}(C,\varOmega _{C})$ vanishing at $P$ and $Q$, but here we give an
explicit construction toward parameterizing the space of these curves.
By a linear transformation, we may assume that
\begin{equation*}
P=(1:0:0:0:0) \quad \text{and} \quad Q=(0:0:0:0:1).
\end{equation*}
Since $\varphi _{i}$ vanishes at $P$ and $Q$, the quadratic forms
$\varphi _{i}$ ($i=1,2,3$) must be of the form
%
\begin{equation}
\label{ThreeQuadForms}
\varphi _{i} =a_{i} \cdot x_{0}x_{4} + f_{i}\cdot x_{0} + g_{i}\cdot x_{4}
+ h_{i}
\end{equation}
with $a_{i} \in K$, where $f_{i}$ and $g_{i}$ are linear forms in
$x_{1},x_{2},x_{3}$, and where $h_{i}$ is a quadratic form in
$x_{1},x_{2},x_{3}$. We shall find an equation only in
$x_{1},x_{2},x_{3}$ from $\varphi _{1}=\varphi _{2}=\varphi _{3}=0$. Put
\begin{equation*}
(v_{1}, v_{2}, v_{3}) := -(h_{1},h_{2},h_{3})\cdot \Delta _{A},
\end{equation*}
where $\Delta _{A}$ is the adjugate matrix of
\begin{equation*}
A :=
\begin{pmatrix}
a_{1} & a_{2} & a_{3}
\\
f_{1} & f_{2} & f_{3}
\\
g_{1} & g_{2} & g_{3}
\end{pmatrix}
.
\end{equation*}
Note that $v_{i}$ ($i=1,2,3$) are polynomials in $x_{1},x_{2},x_{3}$. Since
%
\begin{equation}
\label{Recover}
(v_{1},v_{2},v_{3})=\det (A) (x_{0}x_{4},x_{0},x_{4})
\end{equation}
on $C$, we have the sextic equation
%
\begin{equation}
\label{def:C'}
\det (A)\cdot v_{1}-v_{2}v_{3}=0
\end{equation}
in $x_{1},x_{2},x_{3}$. Let $C'$ be the curve defined by {\eqref{def:C'}} in ${\mathbb P}^{2}=\Proj K[x_{1},x_{2},x_{3}]$. We claim
that this constructs a birational map from $C$ to $C'$. It suffices to
see that $\det A \ne 0$ holds generically, since $x_{0}$ and $x_{4}$ are
recovered from $v_{1},v_{2},v_{3}$ where $\det A \ne 0$ by {\eqref{Recover}}. If $\det A$ were identically zero, then from {\eqref{ThreeQuadForms}} we have $(3-\rank A)$ quadratic forms only in
$x_{1},x_{2},x_{3}$. If $\rank A<2$, then this contradicts that $C$ is
irreducible. If $\rank A = 2$, then let $\psi $ be the quadratic form in
$x_{1},x_{2},x_{3}$; then there is a dominant morphism
$C \to V(\psi ) \subset {\mathbb P}^{2}$, which turns out to be of degree
$2$. This contradicts the assumption that $C$ is not hyperelliptic.

This construction of the sextic $C'$ from $C$ with $P$ and $Q$ has the
following properties.

\begin{prop}%
\label{ConstructionIsCanonical}
\begin{enumerate}[\textup{(3)}]
\item[{\textup{(1)}}] Suppose
$\langle \varphi _{1},\varphi _{2},\varphi _{3} \rangle = \langle
\phi _{1},\phi _{2},\phi _{3}\rangle $ as a linear space over a field defining
$\varphi _{i}$ and $\phi _{i}$ for all $i=1$, $2$, and $3$. Then $C'$ obtained
from $\varphi _{1},\varphi _{2},\varphi _{3}$ is the same as that obtained
from $\phi _{1},\phi _{2},\phi _{3}$.
\item[{\textup{(2)}}] The coordinate change
\begin{equation*}
(x_{0},x_{1},x_{2},x_{3},x_{4}) \mapsto (x_{4}, x_{1},x_{2},x_{3},x_{0})
\end{equation*}
does not change the sextic.
\item[{\textup{(3)}}] The coordinate change
\begin{equation*}
(x_{0},x_{1},x_{2},x_{3},x_{4}) \mapsto (x_{0} + \phi , x_{1},x_{2},x_{3},x_{4}+
\psi )
\end{equation*}
with any linear $\phi ,\psi $ in $x_{1},x_{2},x_{3}$ does not change the
sextic.
\end{enumerate}
\end{prop}
\begin{proof}
(1) Let $B$ be the square matrix of size $3$ such that
$(\varphi _{1},\varphi _{2},\varphi _{3})B=(\phi _{1},\phi _{2},\phi _{3})$.
Let $v_{i}'$ be $v_{i}$ associated to
$\phi _{1},\phi _{2},\phi _{3}$. Then
\begin{equation*}
(v'_{1}, v'_{2}, v'_{3}) = -(h_{1},h_{2},h_{3})B\cdot \det (AB) (AB)^{-1}
= \det (B) (v_{1},v_{2},v_{3}).
\end{equation*}
Hence we have
$\det (AB) v'_{1} - v'_{2}v'_{3} = \det (B) (\det (A)v_{1}-v_{2}v_{3})$.

(2) The matrix $A$ for the new coordinate is
\begin{equation*}
A' :=
\begin{pmatrix}
a_{1} & a_{2} & a_{3}
\\
g_{1} & g_{2} & g_{3}
\\
f_{1} & f_{2} & f_{3}
\end{pmatrix}
.
\end{equation*}
Let $v_{i}'$ be $v_{i}$ associated to the new coordinate. It is straightforward
to see $(v'_{1},v'_{2},v'_{3}) = (-v_{1},v_{3},v_{2})$. Thus the sextic
does not change.

(3) Let $A'$, $h'_{i}$, and $v'_{i}$ be the matrix $A$, the quadratic form
$h_{i}$, and $v_{i}$ for the new coordinate. Then we have
\begin{equation*}
A'=
\begin{pmatrix}
1 & 0 & 0
\\
\psi & 1 & 0
\\
\phi & 0 & 1
\end{pmatrix}
\begin{pmatrix}
a_{1} & a_{2} & a_{3}
\\
f_{1} & f_{2} & f_{3}
\\
g_{1} & g_{2} & g_{3}
\end{pmatrix}
\end{equation*}
and $h'_{i} = h_{i}+a_{i}\phi \psi + f_{i}\phi + g_{j} \psi $. The equation
$\det (A')v'_{1}-v'_{2}v'_{3} = \det (A)v_{1}-v_{2}v_{3}$ is derived from
a straightforward computation, which is tedious but has some beautiful
cancellations.
\end{proof}

From {Proposition~\ref{ConstructionIsCanonical}} (1) and (2), we have the following corollary:

\begin{cor}
\label{cor2.1.2}
Let $K$ be a field. If $C$ and the divisor $P+Q$ on $C$ for distinct two
points $P$ and $Q$ are defined over $K$, then the associated sextic
$C'$ is defined over $K$.
\end{cor}

This corollary suggests us to use the sextic realization to find and/or
enumerate curves over $\Fvtex _{q}$ of genus $5$ with many $\Fvtex _{q^{2}}$-rational
points, since the assumption that $P+Q$ is defined over $\Fvtex _{q}$ is satisfied
if $C$ has many (a few) $\Fvtex _{q^{2}}$-rational points. The case we mainly
treat in this paper is the case of $q=3$.
The maximal number
$\#C(\Fvtex _{9})$ for curves of genus $5$ over $\Fvtex _{9}$ is unknown, as mentioned
in Section~\ref{sec:Intro}. We give an implementation of our algorithm,
restricting ourselves to the case where $C$ is defined over
$\Fvtex _{3}$.

\subsection{Non-special curves of genus 5}
\label{subsec:generic}

Let $C$, $P$, and $Q$ be as in Subsection \ref{subsec:sextic}, especially
$C$ is non-hyperelliptic and non-trigonal. Let $C'$ be the sextic associated
to $(C,P+Q)$ obtained in Subsection \ref{subsec:sextic}. First we see that
$C'$ has singular points, because the arithmetic genus $g(C')$ of
$C'$ is $10$, but $g(C)$ is $5$. Note that $C$ is isomorphic to the desingularization
of $C'$. In general, we have the following formula:
%
\begin{equation}
\label{inequality:genus-multiplicity}
g(C) \le g(C') - \sum _{P'\in C'} \frac{m_{P'}(m_{P'}-1)}{2},
\end{equation}
where $m_{P'}$ is the multiplicity of $C'$ at $P'$.

\begin{lem}%
\label{lem:trigonal-multiplicity}
The sextic $C'$ does not have any singular point of multiplicity $3$, {i.e.},
$m_{P'}\ne 3$.
\end{lem}
\begin{proof}
By a linear transformation we may assume that $(0:0:1)$ is a triple point
of $C'$. Since the defining equation of $C'$ is of degree $3$ as a polynomial
in $z$, the rational map $C' \to {\mathbb P}^{1}$ sending the point $(x:y:z)$ to
$(x:y)$ defines a dominant morphism $C \to {\mathbb P}^{1}$ of degree
$3$. This contradicts the assumption that $C$ is non-trigonal.
\end{proof}
In spite of the lemma, the sextic $C'$ may have various kinds of singularities.
It would be natural that we restrict ourselves to a case that the singularity
is milder. This paper studies the \textit{non-special} case defined just after
the next lemma.

\begin{lem}%
\label{lem:nonsp}
The following are equivalent:
\begin{enumerate}[\textup{(2)}]
\item[\textup{(1)}] The inequality {\eqref{inequality:genus-multiplicity}} becomes
an equality, namely
%
\begin{equation}
\label{GenericGenusFormula}
g(C) = g(C') - \sum _{P'\in C'} \frac{m_{P'}(m_{P'}-1)}{2}.
\end{equation}
\item[\textup{(2)}] $C'$ has distinct five singular points with multiplicity
two.
\end{enumerate}
\end{lem}
\begin{proof}
By $g(C)=5$, in order to have {\eqref{GenericGenusFormula}}, the multiple
set of $m_{P'}$ with $m_{P'} \ge 2$ is either of $\{2,2,2,2,2\}$ and
$\{3,2,2\}$. 
However, $\{3,2,2\}$ does not occur by {Lemma~\ref{lem:trigonal-multiplicity}}.
\end{proof}

\begin{defin}%
\label{def:generic}
We say that $(C,P+Q)$ is \textit{non-special} if the associated sextic model
$C'$ satisfies either (and therefore both) of (1) and (2) of {Lemma~\ref{lem:nonsp}}.
\end{defin}

\begin{rem}
\label{rem2.2.4}
By experiments with computers, we see that in most cases $(C,P+Q)$ is non-special.
A rigorous proof of this will be given in a future paper, as it requires
many pages, cf.\ {Remark~\ref{rem:non-mild}}.
\end{rem}

In the following, we assume that $(C, P+Q)$ is non-special. Let
${P_{1},\ldots , P_{5}}$ be the singular points on the associated
$C'$. If $(C,P+Q)$ is defined over $K$, then the divisor
$P_{1} + \cdots + P_{5}$ is defined over $K$. The absolute Galois group
$G_{K}$ of $K$ makes permutations of $\{P_{1},\ldots , P_{5}\}$. It is
straightforward to see that the pattern of the $G_{K}$-orbits in
$\{P_{1},\ldots , P_{5}\}$ is either of $(1,1,1,1,1)$, $(1,1,1,2)$,
$(1,2,2)$, $(1,1,3)$, $(2,3)$, $(1,4)$, and $(5)$, where for example
$(1,2,2)$ means that $\{P_{1},\ldots , P_{5}\}$ consists of three
$G_{K}$-orbits each of which has cardinality $1$, $2$, and $2$ respectively.
In Section~\ref{sec:4}, in the case of $K=\Fvtex _{3}$ we shall give an explicit
classification of positions of $\{P_{1},\ldots , P_{5}\}$ in
$\mathbb P^{2}$ up to $\PGL _{3}(\Fvtex _{3})$.

Here is another constraint for the position of the singular points. From
now on, we put $x=x_{1}$, $y=x_{2}$, and $z=x_{3}$ for
$x_{1},x_{2},x_{3}$ in Subsection \ref{subsec:sextic} and consider
$K[x,y,z]$ as the coordinate ring of ${\mathbb P}^{2}$.

\begin{prop}%
\label{Reducibility}
If distinct four elements of $\{P_{1},P_{2},P_{3},P_{4},P_{5}\}$ are contained
in a line, then $C'$ is geometrically reducible (i.e.,
$C_{\overline K}:=C \times _{\mathrm{Spec}(K)} \mathrm{Spec}(
\overline{K})$ is reducible, where $\overline K$ is the algebraic closure
of $K$).
\end{prop}
\begin{proof}
Suppose that $P_{1},\ldots , P_{4}$ are contained in a line. By a linear
transformation over $\overline K$, we may assume
\begin{equation*}
P_{1} = (0:0:1),\quad P_{2}=(1:0:1),\quad P_{3}=(c:0:1),\quad P_{4}=(d:0:1),
\end{equation*}
where $0$, $1$, $c$, and $d$ are mutually distinct. Let $f$ be the sextic (in
$x$ and $y$) of $C'$ obtained by substituting $1$ for $z$. Since
$P_{1}$ is a {singular point}, the smallest degree of the non-zero terms
of $f$ is two. Let $a_{i}$ be the $x^{i}y^{0}$-coefficient of $F$ for
$i=2,3,4,5,6$. Then the constant term and the degree-one term of the Taylor
expansion of $f$ at $(x,y)=(b,0)$ are $\sum _{i=2}^{6} b^{i} a_{i}$ and
$\sum _{i=2}^{6} ib^{i-1} a_{i}$ for $b \in \{1,c,d\}$. These are zero,
as $C$ is singular at $P_{i}$ for $i=1,\ldots , 5$. Since
\begin{equation*}
\det
\begin{pmatrix}
1 & 1 & 1 & 1 & 1
\\
c^{2} & c^{3} & c^{4} & c^{5} & c^{6}
\\
2c & 3c^{2} & 4c^{3} & 5 c^{4} & 6 c^{5}
\\
d^{2} & d^{3} & d^{4} & d^{5} & d^{6}
\\
2d & 3d^{2} & 4d^{3} & 5 d^{4} & 6d^{5}
\end{pmatrix}
= c^{4}d^{4}(c-1)^{2}(d-1)^{2}(c-d)^{4},
\end{equation*}
we have that $a_{i}$ are zero for all $i=2, \ldots , 6$. This says that
$f$ is divided by $y$, in particular $C'$ is geometrically reducible.
\end{proof}

Let us study the space of $C'$ having five singular points with multiplicity
$2$. The number of monomials of degree $6$ in three variables is
$28$. For each singular point, we have three linear equations which assure
that the point is singular. Considering a scalar multiplication to the
whole sextic, the number of free parameters is
$28 - 5\times 3 -1 = 12$, see {Proposition~\ref{dim=12}} below for the linear
independence of $5\times 3$ linear equations. This is precisely the dimension
of the moduli space of curves of genus $5$. Note that both of the dimension
of the space of two points on $C$ and that of five points on
${\mathbb P^{2}}$ modulo the action of $\PGL _{3}$ are equal to two. This
says that the parametrization by the sextic models is very effective.

\begin{prop}%
\label{dim=12}
Let $\{P_{1},\ldots ,P_{5}\}$ be distinct five points of
${\mathbb P}^{2}(\overline{K})$. Assume that any distinct four points in
$\{P_{1},\ldots ,P_{5}\}$ are not contained in a line. Then the space of
sextics with double points at $P_{1},\ldots , P_{5}$ up to scalar multiplications
is of dimension $12$.
\end{prop}
\begin{proof}
Renumbering the subscripts of the five points if necessary, we may assume
that $P_{1}$, $P_{2}$, and $P_{3}$ are not in a line. By a linear transformation
by an element of $\PGL _{3}(\overline{K})$, we may also assume that
$P_{1}=(1:0:0)$, $P_{2}=(0:1:0)$, and $P_{3}=(0:0:1)$. Considering the permutation
of $\{x,y,z\}$ by the symmetric group of degree $3$, we may assume that
$P_{4} = (b:c:1)$ and $P_{5}=(d:e:1)$, by the assumption of four points.
Then, any sextic with $z=1$ having singularity at $P_{1}$, $P_{2}$, and
$P_{3}$ is of the form
\begin{eqnarray*}
F&=&a_{1} x^{4} y^{2} + a_{2} x^{4} y+ a_{3} x^{4} + a_{4} x^{3} y^{3}
+ a_{5} x^{3} y^{2} + a_{6} x^{3} y + a_{7} x^{3} + a_{8} x^{2} y^{4} +
a_{9} x^{2} y^{3}+ a_{10} x^{2} y^{2}
\\
&&+ a_{11} x^{2} y + a_{12} x^{2} + a_{13} x y^{4} + a_{14} x y^{3} + a_{15}
x y^{2} {+} a_{16} x y+ a_{17} y^{4} + a_{18}y^{3} + a_{19}y^{2}.
\end{eqnarray*}
The condition that this sextic is singular at $P_{4}$ (resp.
$P_{5}$) is described by three linear equations in
$a_{1},\ldots , a_{19}$ obtained from the condition that $F$,
$\frac{\partial F}{\partial x}$, and $\frac{\partial F}{\partial y}$ are
zero at $P_{4}$ (resp. $P_{5}$). The coefficients of these six linear
equation make a $6\times 19$ matrix $M$. It suffices to show that
$M$ is of rank $6$. By a direct computation, the determinant of the minor
(square matrix of size $6$) corresponding to the coefficients of
$a_{7},a_{11},a_{12},a_{15},a_{18},a_{19}$ is $(be-cd)^{5}(be+cd)$, that of $a_{6}, a_{7}, a_{9}, a_{10}, a_{11}, a_{12}$ is
$b^{6}d^{6}(c-e)^{5}$, and that of
$a_{5}, a_{10}, a_{14}, a_{15}, a_{18}, a_{19}$ is
$c^{6}e^{6}(b-d)^{5}$. If the rank of $M$ were less than $6$, then
the equalities $(be-cd)^{5}(be+cd)=0$, $b^{6}d^{6}(c-e)^{5}=0$, and
$c^{6}e^{6}(b-d)^{5}=0$ have to hold. However, this contradicts the assumption
of four points. Indeed, if $b=0$, then we have $c=0$ or $d=0$ by the first
equation, and then $P_{1}=P_{4}$ holds or
$P_{1}$, $P_{3}$, $P_{4}$, and $P_{5}$ is contained in the line $y=0$. Thus
$b\ne 0$. Similarly $c$, $d$, and $e$ are not zero. Then we have $b=d$ and
$c=e$, which says $P_{3}=P_{4}$. This is absurd. Hence $M$ has to be of
rank $6$.
\end{proof}

%
\begin{rem}%
\label{rem:moduli}
Let ${\mathbf{M}}_{g}$ be the moduli space of curves of genus $g$. Let
$S$ be the subvariety of ${\mathbf{M}}_{5}$ consisting of non-hyperelliptic
and non-trigonal curves. Let ${\mathcal C}' \to \mathcal S$ be the family
of sextic curves $C'$ in ${\mathbb P}^{2}$ whose desingularization
$C\to C'$ satisfies {\eqref{GenericGenusFormula}}. The family
${\mathcal C} \to {\mathcal S}$ obtained by the fiberwise desingularization
of ${\mathcal C}' \to \mathcal S$ defines $f : {\mathcal S} \to S$. It
follows from the construction that $f$ is surjective over every algebraically
closed field, in other words ${\mathcal C} \to {\mathcal S}$ is \textit{geometrically
surjective} in the sense of \cite[Definition 2.1]{LRRS}. The fiber of
$f$ is related to the space of choices of two points on $C$. It would be
interesting to give a precise description of the fiber.
\end{rem}

The next remark explains what {\eqref{GenericGenusFormula}} means.

\begin{rem}%
\label{rem:one-time-blow-up}
We claim that if the desingularization $C \to C'$ satisfies {\eqref{GenericGenusFormula}}, then, at each singular point $P\in C'$, the
desingularization of $C'$ at $P$ is obtained by the one-time blow-up centered
at $P$. Indeed, by a linear coordinate change, the structure ring around
$P$ is locally described as $\Fvtex _{q}[X,Y]/(G)$ with
$G\in \Fvtex _{q}[X,Y]$, where $P$ is the origin. Write
$G = G_{m} + G_{m+1} + \cdots $, where $m=m_{P}$ is the multiplicity at
$P$, and where $G_{i}$ is the homogenous part of degree $i$. By a linear coordinate change
again, we may assume that the $Y^{m}$-coefficient of $G_{m}$ is not zero
if we assume $q-1 \ge m$ (which is satisfied in our case: $q=9$ and
$m\le 3$). Then, the morphism from the strict transform of the blow-up
to $C'$ is locally described by the ring homomorphism
\begin{equation*}
\Fvtex _{q}[X,Y]/(G) \to \Fvtex _{q}[X,Z]/(\tilde G)
\end{equation*}
with ${\tilde G}(X,Z) = G(X,ZX)X^{-m}$ sending $(X,Y)$ to $(X,XZ)$. Its
cokernel is the $\frac{m(m-1)}{2}$-dimensional $\Fvtex _{q}$-vector space generated
by $Z^{i}X^{j}$ for $1\le i \le m-1$ and $0\le j \le i-1$. Thus, we have
the claim.
\end{rem}
Finally, we give a formula on the number of rational points on a non-special
curve $C$ over a finite field $\Fvtex _{q}$.
%
\begin{prop}%
\label{prop:one-time-blow-up}
We have
%
\begin{equation}
\label{formula:NumberOfRationalPoints}
\#C(\Fvtex _{q}) = \#C'(\Fvtex _{q}) + \sum _{P\in C'(\Fvtex _{q})} \left (\#V(h_{P})(
\Fvtex _{q}) - 1\right ),
\end{equation}
where $h_{P}$ is the $G_{m}$ as in {Remark~\ref{rem:one-time-blow-up}} for
$P$, i.e., the homogeneous part of the least degree (i.e., $m_{P}$) of
the Taylor expansion at $P$ of the defining polynomial of an affine model
containing $P$ of $C'$, and where $V(h_{P})$ is the closed subscheme of
${\mathbb P}^{1}$ defined by the ideal $(h_{P})$. Note that
$\#V(h_{P})(\Fvtex _{q}) - 1=0$ if $C'$ is nonsingular at $P$.
\end{prop}
\begin{proof}
Let $\pi : C \to C'$ be the desingularization of $C'$. As $\pi $ maps
$C(\Fvtex _{q})$ to $C'(\Fvtex _{q})$, it suffices to count the $\Fvtex _{q}$-rational
points on the fiber of each point $P \in C'(\Fvtex _{q})$. As in {Remark~\ref{rem:one-time-blow-up}}, by a linear coordinate change, we locally describe
the structure ring around $P$ as $\Fvtex _{q}[X,Y]/(G)$ with
$G\in \Fvtex _{q}[X,Y]$, where $P$ is the origin. By the construction of the
desingularization using blow-ups, $\pi ^{-1}(P)$ is the scheme obtained
by glueing in $\mathbb P^{1}$ the subscheme of ${\mathbb P}^{1}$ defined
by $h_{P}(1,Z)=0$ for $(1:Z)\in {\mathbb P}^{1}$, where
$h_{P}(1,Z)=G_{m}(X,ZX)X^{-m}$ and the subscheme of
${\mathbb P}^{1}$ defined by $h_{P}(W,1)=0$ for
$(W:1)\in {\mathbb P}^{1}$, where $h_{P}(W,1)=G_{m}(WY,Y)Y^{-m}$. (Note
that the fiber of $P$ does not depend on the higher-degree part of
$G$.) Hence $\pi ^{-1}(P)$ is isomorphic to $V(h_{P})$ over
$\Fvtex _{q}$. Clearly the proposition follows from this fact.
\end{proof}

If $h_{P}$ in {Proposition~\ref{prop:one-time-blow-up}} is quadratic, then
$\#V(h_{P})(\Fvtex _{q}) - 1$ is equal to $1$ if the discriminant
$\Delta (h_{P})$ of $h_{P}$ is a nonzero square and to $-1$ if
$\Delta (h_{P})$ is a nonzero non-square and to $0$ if
$\Delta (h_{P}) = 0$. From a computational point of view, it may have a
little advantage to use the following fact:
If $q=p^{2}$ and if $\Delta (h_{P})$ belongs to $\Fvtex _{p}$, then $\Delta (h_{P}) \ne 0$ is equivalent
to that $\Delta (h_{P})$ is a nonzero square in $\Fvtex _{p^{2}}$.

\begin{rem}%
\label{rem:non-mild}
The case of other kinds of singularities (i.e., the number of the singular
points of $C'$ is less than $5$) is one of our future works, see Section~\ref{sec:new6}. As in the non-special case, one could also construct an
algorithm to find a reduced equation defining $C'$ even in other cases,
once the type of the singularities is known. However, the number of singular
points of $C'$ can take any of $5$, $4$, $3$, $2$, and $1$. For example,
consider $C$ defined by
\begin{eqnarray*}
\varphi _{1} &=& x_{0}x_{1} + x_{3}x_{4},
\\
\varphi _{2} &=& x_{0}x_{3} + x_{2}x_{4} + x_{1}^{2},
\\
\varphi _{3} &=& x_{0}x_{4} + x_{2}^{2} + x_{2}x_{3}
\end{eqnarray*}
in characteristic $2$. Then $C'$ is defined by
\begin{equation*}
x_{1}^{5}x_{3} + x_{1}^{2}x_{2}^{4}+x_{1}^{2}x_{2}^{3}x_{3} + x_{2}^{2}x_{3}^{4}+x_{2}x_{3}^{5}.
\end{equation*}
Then $C'$ has a single singular point $(0:1:0)$, whose multiplicity is
two. This means that there are various kinds of singularities, and thus
it would require much effort to classify their possible types.
\end{rem}
%

\section{A concrete algorithm}
\label{sec:3}

Let $K$ be a finite field with characteristic $p$, and $K'$ a finite extension
field of $K$. As it was shown in Subsection \ref{subsec:sextic}, every
non-hyperelliptic and non-trigonal curve of genus $5$ over $K$ is realized
as the normalization of a sextic in $\mathbb{P}^{2}$. In particular, such
a curve $C$ is said to be non-special if it satisfies the condition stated in {Definition~\ref{def:generic}} of Subsection \ref{subsec:generic}.

In this section, we first present a concrete algorithm ({Algorithm~\ref{alg:I}} below) for enumerating all non-special curves $C$ of genus
$5$ over $K$ satisfying $\# C (K') \geq N$, where $N$ is a given positive
integer. After presenting the algorithm, we also give some remarks for
implementation.

\subsection{Algorithm}
\label{subsec:alg}

Let $L$ be a finite extension field of $K$, and let
$\{ P_{1}, P_{2}, P_{3}, P_{4}, P_{5} \}$ be a set of distinct five points
in $\mathbb{P}^{2}(L)$. Assume that $\mathrm{Gal}(L/K)$ stabilizes the
set $\{ P_{1}, P_{2}, P_{3}, P_{4}, P_{5} \}$. Given
$\{ P_{1}, P_{2}, P_{3}, P_{4}, P_{5} \}$ and an integer $N \geq 1$, {Algorithm~\ref{alg:I}} presented below enumerates all sextic forms $F$ in
$K [x,y,z]$ such that the projective scheme $C' : F = 0$ is a singular (irreducible) curve of
geometric genus $5$ in $\mathbb{P}^{2}$ with
$\mathrm{Sing}(C')=\{ P_{1}, P_{2}, P_{3}, P_{4}, P_{5} \}$ and
$\#C (K') \geq N$, where $C$ is the normalization of $C'$.

\begin{al}%
\label{alg:I}
\textit{Input}. $\{ P_{1}, P_{2}, P_{3}, P_{4}, P_{5} \}$, $K'$, and
$N \geq 1$.

\textit{Output}. A set $\mathcal{F}$ of sextic forms $F$ in $K [x,y,z]$ such
that $C' : F = 0$ is a singular (irreducible) curve of geometric genus
$5$ in $\mathbb{P}^{2}$ with
$\mathrm{Sing}(C')=\{ P_{1}, P_{2}, P_{3}, P_{4}, P_{5} \}$ and
$\#C (K') \geq N$, where $C$ is the normalization of $C'$.
\begin{enumerate}[\textup{(5)}]
\item[\textup{(1)}] Set $\mathcal{F} := \emptyset $.
\item[\textup{(2)}] Construct
$F = \sum _{i=1}^{28} a_{i} x^{\alpha _{1}^{(i)}} y^{\alpha _{2}^{(i)}}
z^{\alpha _{3}^{(i)}} \in K [a_{1} , \ldots , a_{28}] [x,y,z]$ with
\begin{equation*}
\{ (\alpha _{1}^{(1)}, \alpha _{2}^{(1)}, \alpha _{3}^{(1)}), \ldots ,
(\alpha _{1}^{(28)}, \alpha _{2}^{(28)}, \alpha _{3}^{(28)}) \} = \{ (
\alpha _{1}, \alpha _{2}, \alpha _{3}) \in (\mathbb{Z}_{\geq 0})^{
\oplus 3} : \alpha _{1} + \alpha _{2} + \alpha _{3} = 6\},
\end{equation*}
where $a_{1}, \ldots , a_{28}$ are indeterminates.
\item[\textup{(3)}] For each $\ell \in \{ 1, \ldots , 5 \}$:
\begin{enumerate}[(a)]
\item[(a)] Write $P_{\ell}=(p_{1}^{(\ell)} : p_{2}^{(\ell)} : p_{3}^{(\ell)})$ for
$p_{1}^{(\ell)}, p_{2}^{(\ell)}, p_{3}^{(\ell)} \in L$.
\item[(b)] Let $k$ be the minimal element in $\{1,2,3 \}$ such that
$p_{k}^{(\ell)} \neq 0$, and put $z_{k} = p_{k}^{(\ell)}$, $z_{i} = X$, and $z_{j}=Y$ for
$i,j\in \{1,2,3\} \smallsetminus \{k\}$ with $i < j$.
\item[(c)] Compute
$F_{\ell}:= F(z_{1},z_{2},z_{3}) \in L[a_{1}, \ldots , a_{28}][X,Y]$. Let
$f_{3 \ell - 2}$ and $f_{3 \ell - 1}$ be the coefficients of $X$ and
$Y$ in $F_{\ell}$ respectively, and $f_{3 \ell}$ the constant term of
$F_{\ell}$ as a polynomial in $X,Y$.
\end{enumerate}
\item[\textup{(4)}] Compute a basis
$\{ \mathbf{b}_{1}, \ldots , \mathbf{b}_{d} \} \subset K^{\oplus 28}$ for
the null-space of the linear system over $L$ defined by
$f_{t}(a_{1}, \ldots , a_{28}) = 0$ with $1 \leq t \leq 15$, where
$d$ denotes the dimension of the null-space. By {Proposition~\ref{dim=12}} we have $d=13$.
\item[\textup{(5)}] For each
$(v_{1}, \ldots , v_{d}) \in K^{\oplus d} \smallsetminus \{ (0,
\ldots , 0) \}$ with $v_{1} \in \{0, 1 \}$:
\begin{enumerate}[(a)]
\item[(a)] Compute
$\mathbf{c}: = \sum _{j=1}^{d} v_{j} \mathbf{b}_{j} \in K^{\oplus 28}
\smallsetminus \{ (0, \ldots , 0) \}$. For each $1 \leq i \leq 28$, we
denote by $c_{i}$ the $i$-th entry of $\mathbf{c}$.
\item[(b)] If
$\sum _{i=1}^{28} c_{i} x^{\alpha _{1}^{(i)}} y^{\alpha _{2}^{(i)}} z^{
\alpha _{3}^{(i)}} \in K[x,y,z]$ is irreducible, and if
$\sum _{i=1}^{28} c_{i} x^{\alpha _{1}^{(i)}} y^{\alpha _{2}^{(i)}} z^{
\alpha _{3}^{(i)}}=0$ in $\mathbb{P}^{2}$ has geometric genus $5$:
\begin{enumerate}[(iii)]
\item[(i)] Set
$F_{\mathbf{c}} := \sum _{i=1}^{28} c_{i} x^{\alpha _{1}^{(i)}} y^{
\alpha _{2}^{(i)}} z^{\alpha _{3}^{(i)}}$, and let $C'$ be the plane curve
in $\mathbb{P}^{2}$ defined by $F_{\mathbf{c}} = 0$.
\item[(ii)] Compute $\# C (K')$ by the formula {\eqref{formula:NumberOfRationalPoints}} given in Subsection \ref{subsec:generic}, where $C$ is the normalization of $C'$.
\item[(iii)] If $\#C (K') \geq N$, replace $\mathcal{F}$ by
$\mathcal{F} \cup \{ F_{\mathbf{c}} \}$.
\end{enumerate}
\end{enumerate}
\item[\textup{(6)}] Output $\mathcal{F}$.
\end{enumerate}
\end{al}

\begin{rem}%
\label{rem:alg}
\begin{enumerate}[(2)]
\item[(1)] In Step (4), we can take a basis
$\{ \mathbf{b}_{1}, \ldots , \mathbf{b}_{d} \}$ so that
$\mathbf{b}_{i} \in K^{\oplus 28}$ for all $1 \leq i \leq d$, by the following
general fact: Let $L/K$ be a separable extension of fields, and $E$ the
Galois closure of $L$ over $K$. If a linear system over $L$ is
$\mathrm{Gal}(E/K)$-stable, then each entry in the Echelon form of the
coefficient matrix of the system belongs to $K$.
\item[(2)] In Step (5)(b)(ii), we need to compute the discriminant
$\Delta _{\ell}$ of the degree-$2$ part of the Taylor expansion of
$F_{\mathbf{c}}$ at each $P_{\ell}$. To achieve this, we can use
$F_{\ell}$ computed in Step (3) as follows: Let
$D_{\ell} \in L[a_{1}, \ldots , a_{28}]$ denote the discriminant of the
degree-$2$ part of $F_{\ell}$ as a polynomial in $X,Y$. Then clearly
we have $\Delta _{\ell} = D_{\ell}(\mathbf{c})$.
\end{enumerate}
\end{rem}

\subsection{Correctness of our algorithm and remarks for implementation}
\label{sec3.2}

The correctness of {Algorithm~\ref{alg:I}} follows mainly from the definition
of multiplicity for a singular point on a projective plane curve: Indeed,
each polynomial $F_{\ell}$ computed in Step (3) is equal to the Taylor
expansion of $F$ at $P_{\ell}$. Since any vector $\mathbf{c} $ defining
$F_{\mathbf{c}}$ in Step (5)(b)(i) is a root of the system
$f_{3 \ell -2} = f_{3 \ell -2} = f_{3 \ell}=0$ for all
$1 \leq \ell \leq 5$, it follows that $P_{\ell}$ is a singular point
of $C' : F_{\mathbf{c}} = 0$ with multiplicity $2$ for each
$1 \leq \ell \leq 5$. Conversely, suppose $\mathbf{c}$ is a vector in
$K^{\oplus 28} \smallsetminus \{ (0, \ldots , 0) \}$ such that
$F_{\mathbf{c}} = 0$ defines an irreducible plane curve of geometric
genus $5$ with multiplicity $2$ at the points
$P_{1}, \ldots , P_{5}$. By the definition of multiplicity for a singularity,
the vector $\mathbf{c}$ is a root of the system
$f_{3\ell -2} = f_{3\ell -1} = f_{3\ell}=0$ for all
$1 \leq \ell \leq 5$. As it was described in {Remark~\ref{rem:alg}} (1),
each entry in the Echelon form of the coefficient matrix of the system
belongs to $K$, and thus $\mathbf{c}$ is a root of a linear system over
$K$ whose null-space over $L$ is the same as that of
$f_{3\ell-2} = f_{3\ell -1} = f_{3\ell}=0$ for all
$1 \leq \ell \leq 5$. Therefore, $\mathbf{c}$ can be expressed as
$\mathbf{c} = \sum _{j=1}^{d} v_{j} \mathbf{b}_{j}$, where
$(v_{1}, \ldots , v_{d}) \in K^{\oplus d} \smallsetminus \{ (0,
\ldots , 0) \}$. Note that it suffices to compute the part with
$v_{1} \in \{0,1\}$, since scalar multiplication does not change the
isomorphism class of the sextic.

We implemented {Algorithm~\ref{alg:I}} over MAGMA V2.25-8~\cite{Magma},
\cite{MagmaHP} in its 64-bit version. Details of our computational environment
will be provided in Section~\ref{sec:new5}. For Step (5)(b), our implementation
makes use of MAGMA's built-in functions such as
\texttt{IsIrreducible}, \texttt{GeometricGenus}, and
\texttt{Variety}. In particular, the function \texttt{Variety} is employed
to compute $\# C' (K') = \# V (F_{\mathbf{c}})$.

\section{Computational results in characteristic 3}
\label{sec4}

This section presents our computational results in characteristic
$p=3$ obtained by applying {Algorithm~\ref{alg:I}}. For this, we first explicitly
determine possible positions in $\mathbb{P}^{2}$ of singular points
of the sextic model associated to a non-special curve of genus $5$ over
$K= \mathbb{F}_{3}$, considering the action by
$\PGL _{3}(\Fvtex _{3})$. With the classification, we execute {Algorithm~\ref{alg:I}} for each position type, and then obtain non-special curves
of genus $5$ over $\mathbb{F}_{3}$ with many $\mathbb{F}_{9}$-rational
points.

\subsection{Position analysis of singular points in the case where $K = \Fvtex _{3}$}
\label{sec:4}

In this subsection, we classify positions of five points
$\{P_{1},\ldots ,P_{5}\}$ in ${\mathbb P}^{2}$ which can be singular points
of our sextic model $C' : F=0$ associated to a non-special curve of genus $5$ over $\Fvtex _{3}$, where we identify two positions
$\{P_{1},\ldots ,P_{5}\}$ and $\{P'_{1},\ldots ,P'_{5}\}$ if there exists
an element $g$ of $\PGL _{3}(\Fvtex _{3})$ that sends $\{P_{1},\ldots ,P_{5}
\}$ to $\{P'_{1},\ldots ,P'_{5}\}$. This is the problem to enumerate
the orbits of an action of a finite group on a finite set. We solved it
by using MAGMA. In the following, we state only the results.

The Frobenius map $\sigma $ over $\Fvtex _{3}$ (defined as the map raising
each entry to the third power) makes a permutation of
$\{P_{1},\ldots , P_{5}\}$, i.e.,
\begin{equation*}
\{P_{1},\ldots , P_{5}\} = \{\sigma (P_{1}),\sigma (P_{2}),\sigma (P_{3}),
\sigma (P_{4}),\sigma (P_{5})\}.
\end{equation*}
The pattern of the Frobenius orbits in $\{P_{1},\ldots , P_{5}\}$ is either
of $(1,1,1,1,1)$, $(1,1,1,2)$, $(1,2,2)$, $(1,1,3)$, $(2,3)$,
$(1,4)$, and $(5)$, where for example $(1,2,2)$ means that
$\{P_{1},\ldots , P_{5}\}$ consists of three Frobenius orbits each of which
has cardinality $1$, $2$, and $2$ respectively.

\paragraph*{Case (1,1,1,1,1)}
This is the case where every singular point is $\Fvtex _{3}$-rational. By a
linear transformation by an element of $\PGL _{3}(\Fvtex _{{3}})$, we may
assume
\begin{equation*}
P_{1}=(1:0:0),\quad P_{2}=(0:1:0),\quad P_{3}=(0:0:1).
\end{equation*}
A computation shows that every position of
$\{P_{1},\ldots , P_{5}\}$ such that any four points among them are not
contained in a line is equivalent by the linear transformation by
a diagonal matrix and a permutation of $\{x,y,z\}$ to either of the following
two cases:

\bigskip\noindent
(1) $P_{4} = (1:1:0)$ and $P_{5} = (0:1:1), \qquad (2)\ P_{4} = (1:1:0)$ and $P_{5} = (1:2:1)$.

\begin{ex}
\label{exmp4.1.1}
In the case of characteristic $3$, the defining equation of any sextic having singular
points $P_{1}=(0:0:1)$, $P_{2}=(0:1:0)$, $P_{3}=(1:0:0)$
$P_{4}=(1:1:0)$, $P_{5}=(0:1:1)$ is
\begin{eqnarray*}
{F}&=& a_{1} x^{4}y^{2} +a_{2} x^{4}yz +a_{3} x^{4} z^{2} +a_{4} x^{3}
y^{3} + a_{5} x^{3} y^{2} z + a_{6} x^{3} y z^{2} + a_{7} x^{3} z^{3}
\\
&& + a_{8} x^{2} y^{4} + a_{9} x^{2} y^{3} z + a_{10} x^{2} y^{2} z^{2}
+ a_{11} x^{2} y z^{3} + a_{12} x^{2} z^{4} + a_{13} x y^{4} z
\\
&& + a_{14} x y^{3} z^{2} + a_{15} x y^{2} z^{3} + a_{16} x y z^{4} + a_{17}
y^{4} z^{2} + a_{18} y^{3} z^{3} + a_{19} y^{2} z^{4}
\end{eqnarray*}
with
\begin{equation*}
\begin{cases}
a_{1} + a_{4} + a_{8} = 0,
\\
2a_{1}+a_{8} = 0,
\\
a_{2}+a_{5}+a_{9}+a_{13} = 0,
\end{cases}
\quad
\begin{cases}
a_{17}+a_{18}+a_{19} = 0,
\\
a_{13} + a_{14} + a_{15} + a_{16} =0,
\\
a_{17}+2a_{19}=0.
\end{cases}
\end{equation*}
The quadratic form $h_{P}(=G_{2})$ with the notation of {Remark~\ref{rem:one-time-blow-up}} and {Proposition~\ref{prop:one-time-blow-up}} associated to the singular point
$P_{i}$ (of multiplicity $2$) is as follows:
\begin{eqnarray*}
{h_{P_{1}}} &:=& a_{12} x^{2} + a_{16} xy + a_{19} y^{2},
\\
{h_{P_{2}}} &:=& a_{8} x^{2} + a_{13} xz + a_{17} z^{2},
\\
{h_{P_{3}}} &:=& a_{1} y^{2} + a_{2} yz + a_{3} z^{2},
\\
{h_{P_{4}}} &:=& a_{1}(y-x)^{2} + (a_{2} + 2a_{5} + a_{13})(y-x)z +(a_{3}
+ a_{6} + a_{10} + a_{14} + a_{17})z^{2},
\\
{h_{P_{5}}} &:=& (a_{8} + a_{9} + a_{10} + a_{11} + a_{12})x^{2} + (a_{13}
+ 2 a_{14} + a_{16})x(y-z) + a_{17} (y-z)^{2}
\end{eqnarray*}
respectively.
\end{ex}
%

\paragraph*{Case (1,1,1,2) with linearly independent $P_{1}$, $P_{2}$, $P_{3}$}
We consider the case where $\{P_{1}, \ldots , P_{5}\}$ contains three
$\Fvtex _{{3}}$-rational points, say $P_{1},P_{2},P_{3}$, where $P_{1}$,
$P_{2}$, and $P_{3}$ are linearly independent, and the other points are defined
over $\Fvtex _{{{3}}^{2}}$ (not over $\Fvtex _{{3}}$) and are conjugate to each
other, i.e., $(P_{4},P_{5})=(\sigma (P_{5}),\sigma (P_{4}))$. By a transformation
by an element of $\PGL _{3}(\Fvtex _{{3}})$, we may assume
\begin{equation*}
P_{1}=(1 : 0 : 0),\quad P_{2}=(0 : 1 : 0),\quad P_{3}=(0 : 0 : 1).
\end{equation*}
Let $\zeta $ be a primitive element of $\Fvtex _{{{3}}^{2}}$. A computation
shows that every position of $\{P_{1},\ldots , P_{5}\}$ such that any four
points among them are not contained in a line is equivalent by a linear
transformation by a diagonal matrix and a permutation of
$\{x,y,z\}$ to either of the three cases:

\bigskip\noindent
$(1)\ P_{4} = (1 : \zeta ^{5} : \zeta ^{7}),\qquad (2)\
P_{4} = (1 : \zeta ^{7} : 1),\qquad (3)\
P_{4} = (1 : \zeta ^{2} : \zeta ^{2})$

\bigskip\noindent
with $P_{5} = \sigma (P_{4})$.

\paragraph*{Case (1,1,1,2) with linearly dependent $P_{1}$, $P_{2}$, $P_{3}$}
We consider the case where $\{P_{1},P_{2},P_{3},P_{4},P_{5}\}$ contains
three $\Fvtex _{{3}}$-rational points, say $P_{1},P_{2},P_{3}$, where
$P_{1}$, $P_{2}$, and $P_{3}$ are linearly dependent, and the other points
are defined over $\Fvtex _{{{3}}^{2}}$ and are conjugate to each other.
By a transformation by an element of $\PGL _{3}(\Fvtex _{{3}})$, we may assume
\begin{equation*}
P_{1}=(1 : 0 : 0),\quad P_{2}=(0 : 1 : 0),\quad P_{3}=(1:1:0).
\end{equation*}
Let $\zeta $ be a primitive element of $\Fvtex _{{{3}}^{2}}$.
As in the case $(1,1,1,2)$ with linearly independent $P_1,P_2,P_3$, we have three
equivalent classes:

\bigskip\noindent
$(1)\ P_{4} = (1:\zeta ^{5}:\zeta ^{7}),\qquad
(2)\ P_{4} = (1:1:\zeta ^{7}),\qquad
(3)\ P_{4} = (1:\zeta ^{2};\zeta ^{2})$

\bigskip\noindent
with $P_{5} = \sigma (P_{4})$.

\paragraph*{Case (1,2,2)}
We consider the case where $\{P_{1},P_{2},P_{3},P_{4},P_{5}\}$ contains
one $\Fvtex _{{3}}$-rational point $P_{1}$ and two pairs $(P_{2},P_{3})$ and
$(P_{4},P_{5})$ of $\Fvtex _{{{3}}^{2}}$-rational points, where
$P_{3}=\sigma (P_{2})$ and $P_{5} = \sigma (P_{4})$. By a transformation
by an element of $\PGL _{3}(\Fvtex _{{3}})$, we may assume
\begin{equation*}
P_{1}=(1:0:0).
\end{equation*}
Let $\zeta $ be a primitive element of $\Fvtex _{{{3}}^{2}}$. We have five
equivalent classes:

\bigskip\noindent
(1) $(P_{2},P_{4})= ((1:1:\zeta ^{2}),(0:1:\zeta ^{6})), \qquad
(2)\ (P_{2},P_{4})= ((1:2:\zeta ^{5}),(1:\zeta ^{2}:\zeta ^{7}))$,

\noindent
(3) $(P_{2},P_{4})= ((1: \zeta : 1),(1:\zeta ^{7}:\zeta ^{7})),
\qquad (4)\ (P_{2},P_{4})= ((1:\zeta ^{2}:\zeta ^{6}),(1:0:\zeta ^{5}))$,

\noindent
(5) $(P_{2},P_{4})= ((1:0:\zeta ^{2}),(1:1:\zeta ^{5}))$.

\paragraph*{Case (1,1,3)}
We consider the case where $\{P_{1},P_{2},P_{3},P_{4},P_{5}\}$ contains
two $\Fvtex _{{3}}$-rational points $P_{1}, P_{2}$ and conjugate three points
$(P_{3},P_{4},P_{5})$ of $\Fvtex _{{{3}}^{3}}$-rational points, where
\begin{equation*}
(P_{3},P_{4},P_{5})=(P_{3},\sigma (P_{3}), \sigma ^{2}(P_{3})).
\end{equation*}
By a transformation by an element of $\PGL _{3}(\Fvtex _{{3}})$, we may assume
\begin{equation*}
P_{1}=(1:0:0),\quad P_{2}=(0:1:0).
\end{equation*}
Let $\zeta $ be a primitive element of $\Fvtex _{{{3}}^{3}}$. We have four
cases

\bigskip\noindent
(1) $P_{3}=(1:2:\zeta ^{5}), \quad (2)\
P_{3}=(1:\zeta ^{6}:\zeta ^{25}), \quad (3)\
P_{3}=(1:\zeta ^{17}:\zeta ^{2}),\quad (4)\
P_{3}=(1:2:\zeta ^{10})$.

\paragraph*{Case (2,3)}
We consider the case where
$\{P_{1},P_{2},P_{3},P_{4},P_{5}\}$ contains conjugate two points
$(P_{1},P_{2})$ of $\Fvtex _{{{3}}^{2}}$-rational points with
$(P_{1},P_{2})=(P_{1},\sigma (P_{1}))$ and conjugate three points
$(P_{3},P_{4},P_{5})$ of $\Fvtex _{{{3}}^{3}}$-rational points, where
$(P_{3},P_{4},P_{5})=(P_{3},\sigma (P_{3}), \sigma ^{2}(P_{3}))$. By a
transformation by an element of $\PGL _{3}(\Fvtex _{{3}})$, we may assume
\begin{equation*}
P_{1}=(1: \xi : 0){,}
\end{equation*}
where $\xi $ is a primitive element of $\Fvtex _{{{3}}^{2}}$. Let
$\zeta $ be a primitive element of $\Fvtex _{{{3}}^{3}}$. We have three cases:

\bigskip\noindent
(1) $P_{3} = (1:2:\zeta ^{5}), \qquad (2)\
P_{3} = (1:\zeta ^{22}:2), \qquad (3)\
P_{3} = (1:\zeta ^{17}:\zeta ^{2})$.

\paragraph*{Case (1,4)}
We consider the case where $\{P_{1},P_{2},P_{3},P_{4},P_{5}\}$ contains
one $\Fvtex _{3}$-rational point $P_{1}$ and the other 4 points are over
$\Fvtex _{3^{4}}$ and are conjugate to each other, say
\begin{equation*}
(P_{2},P_{3},P_{4},P_{5})=(P_{2},\sigma (P_{2}), \sigma ^{2}(P_{2}),
\sigma ^{3}(P_{2})).
\end{equation*}
By a transformation by an element of $\PGL _{3}(\Fvtex _{3})$, we may assume
$P_{1}=(1:0:0)$. Let $\zeta $ be a primitive element of
$\Fvtex _{3^{4}}$. We have five equivalent classes:

\bigskip\noindent
(1) $P_{2} = (1:\zeta ^{75}:\zeta ^{49}), \qquad (2)\
P_{2} = (1:\zeta ^{8}:\zeta ^{70}), \qquad (3)\
P_{2} = (1:\zeta ^{59}:\zeta ^{53})$,

\noindent
(4) $P_{2} = (1:\zeta ^{72}:\zeta ^{29}), \qquad (5)\
P_{2} = (1:\zeta ^{5}:\zeta ^{75})$.

\paragraph*{Case (5)}
This is the case where singular points on $C'$ are defined over
$\Fvtex _{3^{5}}$ (but not over $\Fvtex _{3}$). Then
$\{P_{1}, \ldots , P_{5}\}$ consists of a single Frobenius orbit, namely
\begin{equation*}
\{P_{1}, P_{2}, P_{3}, P_{4}, P_{5}\}=\{P_{1}, \sigma (P_{1}),
\sigma ^{2}(P_{1}), \sigma ^{3}(P_{1}),\sigma ^{4}(P_{1})\}.
\end{equation*}
In this case, the five points are determined only by $P_{1}$. A computation
says that there are two cases:

\bigskip\noindent
(1) $P_{1} = (1: \zeta ^{127}: \zeta ^{143}), \qquad
(2)\ P_{1} = (1: \zeta ^{218}: \zeta ^{72})$

\bigskip\noindent
for a primitive element $\zeta $ of $\Fvtex _{3^{5}}$.

\subsection{Non-special curves of genus 5 over $\mathbb{F}_{3}$ with many $\mathbb{F}_{9}$-rational points}
\label{sec:new5}

For $K=\Fvtex _{p}$ and $K'=\Fvtex _{p^{2}}$ with $p=3$, we executed {Algorithm~\ref{alg:I}} over MAGMA V2.25-8 in each case given in Section~\ref{sec:4}, in order to prove {Theorem~\ref{thm:main}}. In our computations,
we choose $N=32$ as the input, since $32$ is the maximal number among
the known numbers of $\Fvtex _{9}$-rational points of genus-five curves over
$\Fvtex _{9}$. Our implementations over MAGMA were conducted on a PC equipped
with an Ubuntu 18.04.5 LTS OS, utilizing a 3.50GHz quad-core CPU
(Intel(R) Xeon(R) E-2224G) and 64GB of RAM. It took about 58 hours in
total to execute the algorithm for obtaining {Theorem~\ref{thm:main}}. For
curves $C$ with $\# C (\mathbb{F}_{{{3}}^{2}}) \ge 32$ listed up, we also
computed their Weil polynomials over MAGMA. We have also obtained explicit
equations for non-special curves of genus 5 over
$\mathbb{F}_{3}$ having $32$ $\mathbb{F}_{9}$-rational points. Here,
we show some examples:
\begin{enumerate}[(5)]
\item[(1)] Case $(1,1,1,1,1)$ with $P_{5}=(1:2:1)$. The normalization $C$ of the sextic $C':F=0$ with
\begin{eqnarray*}
F &=&x^{4}y^{2} + x^{3}y^{3} + x^{2}y^{4} + x^{4}yz + x^{3}y^{2}z + 2x^{2}y^{3}z
+ 2xy^{4}z + 2x^{2}yz^{3} + xy^{2}z^{3} + x^{2}z^{4} \\
&&{}+ 2xyz^{4} + y^{2}z^{4}
\end{eqnarray*}
has 32 $\Fvtex _{9}$-rational points with Weil polynomial
$(t + 3)^{6}(t^{2} + 2t + 9)^{2}$.
\item[(2)] Case $(1,1,1,2)$ with linearly independent $P_{1},P_{2},P_{3}$ where
$P_{4}=(1:\zeta ^{2}:\zeta ^{2})$ for a primitive element $\zeta$ of $\mathbb{F}_{3^2}$.
The normalization $C$ of the sextic $C':F=0$ with
\begin{equation*}
F=x^{2}y^{4} + x^{4}yz + 2y^{4}z^{2} + x^{2}z^{4} + 2y^{2}z^{4}
\end{equation*}
has 32 $\Fvtex _{9}$-rational points with Weil polynomial
$(t^{2} + 2t + 9)(t^{2} + 5t + 9)^{4}$.
\item[(3)] Case $(1,1,1,2)$ with linearly dependent $P_{1},P_{2},P_{3}$ where
$P_{4}=(1:\zeta ^{5}:\zeta ^{7})$ for a primitive element $\zeta$ of $\mathbb{F}_{3^2}$.
The normalization $C$ of the sextic $C':F=0$ with
\begin{eqnarray*}
F&=&x^{4}y^{2} + x^{3}y^{3} + x^{2}y^{4} + 2x^{3}y^{2}z + xy^{4}z + x^{2}y^{2}z^{2}
+ 2xy^{3}z^{2} + 2x^{3}z^{3}
\\
&& + 2y^{3}z^{3} + x^{2}z^{4} + 2xyz^{4} + 2y^{2}z^{4} + z^{6}
\end{eqnarray*}
has 32 $\Fvtex _{9}$-rational points with Weil polynomial
$(t+3)^{4}(t^{2} + 2t + 9)(t^{2}+4t+9)^{2}$.
\item[(4)] Case $(1,2,2)$ with $P_{2}=(1:2:\zeta ^{5})$ and
$P_{4}=(1:\zeta ^{2}:\zeta ^{7})$, where $\zeta$ is a primitive element of $\mathbb{F}_{3^2}$.
The normalization $C$ of the sextic $C':F=0$ with
\begin{eqnarray*}
F&=&x^{4}y^{2} + 2x^{3}y^{3} + 2xy^{5} + 2y^{6} + x^{2}y^{3}z + 2y^{5}z
+ 2x^{4}z^{2} + x^{3}yz^{2} + xy^{3}z^{2}
\\
&&+ 2x^{3}z^{3} + x^{2}yz^{3} + xyz^{4} + y^{2}z^{4} + 2xz^{5} + 2yz^{5}
+ z^{6}
\end{eqnarray*}
has 32 $\Fvtex _{9}$-rational points with Weil polynomial
$(t + 3)^{2} (t^{4} + 8t^{3} + 32t^{2} + 72t + 81)^{2}$.
\item[(5)] Case $(1,1,3)$ with $P_{3}=(1:2:\zeta ^{5})$, where $\zeta$ is a primitive element of $\mathbb{F}_{3^3}$.
The normalization $C$ of the sextic $C':F=0$ with
\begin{eqnarray*}
F&=&x^{4}y^{2} + x^{2}y^{4} + x^{4}yz + 2xy^{4}z + 2x^{4}z^{2} + 2x^{3}yz^{2}
+ x^{2}y^{2}z^{2} + 2xy^{3}z^{2} + 2y^{4}z^{2} + 2x^{3}z^{3}
\\
&& + 2x^{2}yz^{3} + xy^{2}z^{3} + y^{3}z^{3} + 2x^{2}z^{4} + xyz^{4} +
2y^{2}z^{4} + xz^{5} + 2yz^{5} + 2z^{6}
\end{eqnarray*}
has 32 $\Fvtex _{9}$-rational points with Weil polynomial
$(t^{2} + 2t + 9)(t^{2} + 5t + 9)^{4}$.
\item[(6)] Case $(2,3)$ with $P_{1}=(1:\xi :0)$ and
$P_{3}=(1:2:\zeta ^{5})$, where $\xi$ is a primitive element of $\mathbb{F}_{3^2}$ and where $\zeta$ is a primitive element of $\mathbb{F}_{3^3}$.
The normalization $C$ of the sextic $C' : F=0$ with
\begin{eqnarray*}
F &=& x^{5}y + x^{4}y^{2} + 2x^{2}y^{4} + 2y^{6} + x^{5}z + x^{4}yz + 2x^{3}y^{2}z
+ 2x^{3}yz^{2} + x^{2}y^{2}z^{2} + x^{3}z^{3}
\\
& & + 2x^{2}yz^{3} + xy^{2}z^{3}+ 2x^{2}z^{4} + xyz^{4} + 2y^{2}z^{4} +
xz^{5} + z^{6}
\end{eqnarray*}
has 32 $\Fvtex _{9}$-rational points with Weil polynomial
$(t^{2} + 2t + 9)(t^{2} + 5t + 9)^{4}$.
\end{enumerate}

\section{Concluding remarks with some open problems}
\label{sec:new6}

We provided a new effective parametrization for the space of curves of
genus $5$ that are neither hyperelliptic nor trigonal. We realized such
a curve as the normalization of a sextic in $\mathbb{P}^{2}$, and termed
a non-special curve if the associated sextic had five double points. It
was also proved that the dimension of the space of non-special curves of
genus $5$ does not exceed $12$, where the value $12$ is equal to the dimension
of the moduli space of curves of genus $5$. As an application, we also
presented an algorithm for enumerating non-special curves of genus
$5$ over finite fields with more rational points than a prescribed bound.
By executing the algorithm
on MAGMA, we showed that the maximal number of $\mathbb{F}_{9}$-rational
points of non-special curves of genus $5$ over $\mathbb{F}_{3}$ is
$32$.

Our explicit parametrization and the algorithm presented in this paper
may find fruitful applications both in theory and in computation. One of such
applications is the classification of possible invariants (Hasse-Witt rank
and so on) for non-hyperelliptic and non-trigonal curves of genus
$5$. Finally, we list some considerable open problems:
\begin{enumerate}[(d)]
\item[(a)] Extend the parametrization to the case where curves have more
complex singularities. Concretely, present a parametrization for non-hyperelliptic
and non-trigonal curves $C$ of genus $5$, such that the equality {\eqref{GenericGenusFormula}} in Subsection \ref{subsec:generic} does not
hold (cf.\ {Remark~\ref{rem:non-mild}}). A more general problem is to give
an explicit model in $\mathbb{P}^{2}$ for non-hyperelliptic curves of the
other genus $\geq 4$. Cf.\ a parametrization of generic curves of genus
$3$ is presented in \cite{LRRS}, where an equation with $7$ parameters
(the moduli dimension is $6$) is given. We also refer to \cite{BG}, where
the authors give representative families for all strata by automorphism
group (except for the stratum associated with automorphism group
$\mathbb{Z}/5\mathbb{Z}$) of smooth plane curves of genus $6$.
\item[(b)] Present methods for computing invariants of non-special curves
of genus $5$, such as Cartier-Manin and Hasse-Witt matrices. Computing
Cartier-Manin and Hasse-Witt matrices enables us to determine whether given
curves are superspecial or not, as in \cite{KH16}, \cite{KH18},
\cite{trigonal}, and \cite{KH17}. To compute Cartier-Manin matrices, a method
provided in \cite{SV} for a plane curve can be applied.
\item[(c)] Construct an algorithm for determining whether given two non-special
curves of genus $5$ are isomorphic or not. With the algorithm, we might
present a characterization of the moduli space of curves of genus
$5$, employing the concept of representative families as described in
\cite{LRRS} (cf.\ {Remark~\ref{rem:moduli}}). Computing the automorphism
group of such a curve is also an interesting problem. Cf.\ in the case
of non-hyperelliptic curves of genus 4, the authors (resp. the authors
and Senda) presented an algorithm for the isomorphism test (resp. for
computing automorphism groups) in \cite{KH17} (resp. \cite{KHS17}).
\item[(d)] Improve the efficiency of {Algorithm~\ref{alg:I}}, and then enumerate
non-special curves of genus $5$ having many rational points for
$p>3$. To accomplish this, it is important to reduce the number of curves
in the search space.
\end{enumerate}

\footnote[0]{
E-mail address of the first author: \texttt{m-kudo@fit.ac.jp}\\
E-mail address of the second author: \texttt{harasita@ynu.ac.jp}
}

\end{document}